\documentclass{article}
\usepackage{euscript,amsmath,amssymb,amsfonts,graphicx,bm,amsthm,mathtools, bbm}
  \usepackage{caption}
\usepackage{subcaption}
\usepackage{pdflscape}

\usepackage{array}

\usepackage{cite} 

  \usepackage{paralist}
  \usepackage{graphics} 
\usepackage{soul}

\usepackage{latexsym,amssymb,enumerate,amsmath,amsthm} 
  \usepackage{graphicx} 
  \usepackage{tikz}
 \usepackage[hidelinks]{hyperref}
\usepackage{latexsym,amssymb,enumerate,amsmath} 
\usepackage{pgfplots}
\pgfplotsset{compat=1.18}
\usepackage{soul,xcolor}

\usepackage{lineno}
\newcommand*\patchAmsMathEnvironmentForLineno[1]{%
  \expandafter\let\csname old#1\expandafter\endcsname\csname #1\endcsname
  \expandafter\let\csname oldend#1\expandafter\endcsname\csname end#1\endcsname
  \renewenvironment{#1}%
     {\linenomath\csname old#1\endcsname}%
     {\csname oldend#1\endcsname\endlinenomath}}%
\newcommand*\patchBothAmsMathEnvironmentsForLineno[1]{%
  \patchAmsMathEnvironmentForLineno{#1}%
  \patchAmsMathEnvironmentForLineno{#1*}}%
\AtBeginDocument{%
\patchBothAmsMathEnvironmentsForLineno{equation}%
\patchBothAmsMathEnvironmentsForLineno{align}%
\patchBothAmsMathEnvironmentsForLineno{flalign}%
\patchBothAmsMathEnvironmentsForLineno{alignat}%
\patchBothAmsMathEnvironmentsForLineno{gather}%
\patchBothAmsMathEnvironmentsForLineno{multline}%
}

\newtheorem{theorem}{Theorem}

\newtheorem{lemma}{Lemma}

\theoremstyle{plain}
\theoremstyle{remark}

\theoremstyle{definition}

\newcommand{\dd}{\textup{d}}
\def\eps{\varepsilon}
\def\E{\mathbb{E}}
\def\P{\mathbb{P}}
\def\R{\mathbb{R}}
\newcommand{\Markov}[2]{\underset{#1}{\overset{#2}{\rightleftharpoons}}}
\def\tmin{t_{\min}}
\def\bm{\mathrm{bm}}
\def\disc{\mathrm{disc}}
\def\rtp{\mathrm{rtp}}

\def\Erlang{\text{Erlang}}
\def\when{\text{ when }}
\def\withProb{\text{ with probability }}
\def\wtlambda{\widetilde{\lambda}}
\def\wttau{\widetilde{\tau}}
\def\wtX{\widetilde{X}}
\def\wtTk{\widetilde{T_k}}
\def\mcTk{\mathcal{T}_k}

\begin{document}


\title{Search at bounded speed with immigration}


\author{Hwai-Ray Tung\thanks{Program in Applied and Computational Mathematics, Princeton University, Princeton, NJ 08544 USA} \thanks{Current Affiliation: Department of Mathematics, Texas A\&M University, College Station, TX 77840 USA}
\and
Sean D. Lawley\thanks{Department of Mathematics, University of Utah, Salt Lake City, UT 84112 USA (\texttt{lawley@math.utah.edu}).}
}
\date{\today}
\maketitle

\begin{abstract}
Many biophysical search processes employ searchers which enter or ``immigrate'' into the domain progressively over time. Existing search time estimates can become unphysical for fast immigration since they imply that searchers move at infinite speed. In this paper, we investigate search times of immigrating searchers that move at bounded speed. In the fast immigration limit, we determine the full probability distribution and all the moments of the $k$th search time in terms of the early time probability distribution of a single searcher. We apply these rigorous mathematical results to several canonical models of stochastic search. We further analyze different models of ``diffusion'' and use these results to investigate when and how the minutiae of searcher dynamics affect search times. We compare our theory to numerical simulations.
\end{abstract}


\section{\label{sec:introduction}Introduction}

The timescales of many natural and engineered processes can be formulated in terms of first passage times (FPTs) \cite{grebenkov2024target}. A FPT generically describes the time it takes a random searcher to find a target. Examples include the time it takes a predator to capture a prey, a virus to infect a host, a calcium ion to bind a receptor, and a sperm cell to fertilize an egg. In many systems, including the aforementioned examples, the important timescale is the time it takes the fastest searcher(s) to find the target out of many parallel searchers \cite{lawley2024competition}. Such ``extreme FPTs'' or ``fastest FPTs'' (fFPTs) are commonly modeled by assuming that there are $N\gg1$ searchers initially present in the system \cite{weiss1983, yuste1996, yuste1997, yuste2001, meerson2015}.

More recently, fFPTs have been studied in models where searchers are added progressively over time \cite{campos2024dynamic, tung2025first, grebenkov2025fastest, tung2025passage, meyer2025optimal, linn2026dynamic}. Introduced by Campos and M\'endez \cite{campos2024dynamic}, the simplest model assumes that searchers enter the system (or ``immigrate'' or are ``born'') at a constant Poissonian rate $\lambda>0$. For the prototypical model of diffusive search in which searchers move by Brownian motion, the mean {fFPT} in the fast immigration limit satisfies
\begin{align}\label{eq:fastimmigration0}
    \E[T_1]
    \sim\frac{L^2}{4D\ln(\lambda L^2/(4D))}\quad\text{as }\lambda L^2/D\to\infty,
\end{align}
where $D>0$ is the searcher diffusivity and $L>0$ is the shortest distance a searcher must travel to find the target. Throughout this paper, $f\sim g$ denotes $\lim f/g=1$. 
The asymptotic in \eqref{eq:fastimmigration0} was first shown for pure diffusion processes in one-dimension (1D) by Campos and M\'endez \cite{campos2024dynamic} and later extended to more general diffusion processes in arbitrary spatial dimension in Ref.~\cite{tung2025first}. Rather than continuous diffusion, if searchers move on a discrete network of states (i.e.\ a graph of nodes), then the mean {fFPT} satisfies \cite{tung2025first}
\begin{align}\label{eq:fastimmigration0discrete}
    \E[T_1]
    \sim 
    \frac{B(d)}{{{r}}}\Big(\frac{{{r}}}{\lambda}\Big)^{1/(d+1)}\quad\text{as }\lambda/{{r}}\to\infty,
\end{align}
where searchers jump at rate ${{r}}>0$ between states and $d\ge1$ is the smallest number of jumps a searcher must take to reach the target. If there is a single path of $d\ge1$ states connecting the starting location to the target, then the prefactor is $B(d)=((d+1)!)^{1/(d+1)}\Gamma(1+1/(d+1))\approx (d+1)/e$, where the approximation uses Stirling's formula and is accurate for large $d$.

The physical relevance of \eqref{eq:fastimmigration0} and \eqref{eq:fastimmigration0discrete} might be questioned since these asymptotics imply that searchers can move with unbounded speed. Indeed, if the searcher speed is bounded above by $v<\infty$ and searchers must travel distance $L>0$ to reach the target, then 
\begin{align}\label{eq:bound0}
    T_1\ge \tmin
    :=L/v>0.
\end{align}
However, the asymptotic \eqref{eq:fastimmigration0} contradicts the bound \eqref{eq:bound0} for sufficiently fast immigration, i.e.\ if
\begin{align*}
    \lambda L^2/D>4e^{vL/(4D)}.    
\end{align*}
Similarly, the asymptotic \eqref{eq:fastimmigration0discrete} contradicts the bound \eqref{eq:bound0} for sufficiently large $\lambda/{{r}}$. Therefore, the asymptotics in \eqref{eq:fastimmigration0} and \eqref{eq:fastimmigration0discrete} must fail at sufficiently large immigration rates if searchers move with bounded speed.

The purpose of this paper is to investigate {fFPT}s for the process where searchers immigrate into the system at rate $\lambda>0$ and then move with bounded speed. The bounded speed assumption means that once a searcher enters the system, the additional time it takes that searcher to find the target is a random variable $\tau$ bounded away from zero, i.e.\ 
\begin{align}\label{eq:boundedspeed0}
    \P(\tau<\tmin)=0,
\end{align}
for some strictly positive minimal travel time $\tmin>0$. We assume that searchers are independent after entering the system. Mathematically, the {fFPT} is thus the following minimum,
\begin{align}\label{eq:T10}
    T_1
    =\min_n\bigg\{\tau_n+\frac{1}{\lambda}\sum_{i=1}^n \sigma_i\bigg\},
\end{align}
where $\{\sigma_i\}_{i\ge1}$ is an independent and identically distributed (iid) sequence of unit rate exponential random variables and $\{\tau_n\}_{n\ge1}$ is an iid sequence of the random variable $\tau$ in \eqref{eq:boundedspeed0}. In words, the $n$th time in the set on the right-hand side of \eqref{eq:T10} is the search time after entering the system ($\tau_n$) plus the $n$th immigration time ($\frac{1}{\lambda}\sum_{i=1}^n \sigma_i$). The model formulation is detailed in section~\ref{sec:model}.

In section~\ref{sec:math}, we determine the probability distribution and all the moments of the {fFPT} $T_1$ (and the $k$th {fFPT} $T_k$) in the fast immigration limit (i.e.\ as $\lambda\to\infty$). These rigorous mathematical results are general theorems that depend on the behavior of the probability distribution of $\tau$ near $\tmin$. These results do not fit the mold of classical extreme value theory, since $T_1$ in \eqref{eq:T10} is the minimum of an infinite sequence of random variables which are neither independent nor identically distributed. These results are a rare instance in which extreme value statistics and distributions can be derived for non-iid random variables \cite{majumdar2020}.

In section~\ref{sec:examples}, we apply these results to several examples, including searchers which move by a run-and-tumble process in 1D, two-dimensions (2D), and three-dimensions (3D). We also compare our results to numerical simulations.

In section~\ref{sec:diffusion}, we use these results to investigate the diffusion asymptotic in \eqref{eq:fastimmigration0} for bounded speed searchers. In particular, we consider searchers who move by a bounded speed analog of diffusion (namely, a run-and-tumble process with speed $v=v_0/\eps$ and tumbling rate $r=r_0/\eps^2$ where $0<\eps\ll1$ and $D=v_0^2/(2r_0)$). We find that the asymptotic in \eqref{eq:fastimmigration0} approximates the mean {fFPT} if $1\ll \lambda L^2/D\le \chi_1$, whereas the bounded speed theory of section~\ref{sec:math} describes the {fFPT} for $\lambda L^2/D\gg\chi_2$, for some thresholds $\chi_1\le\chi_2$ which we estimate in section~\ref{sec:diffusion}.  

We conclude with a brief discussion in section~\ref{sec:discussion}. An Appendix collects the proofs of section~\ref{sec:math}, numerical simulation methods, and some technical calculations.

\section{\label{sec:model}Model}
In this section, we detail the model of search with immigration. Searchers are added at a constant Poissonian rate $\lambda$ as in Refs.~\cite{campos2024dynamic, tung2025first}, except we start with zero searchers instead of one. More concretely, if $\{\sigma_i\}_i$ are iid unit rate exponential random variables, then the $n$th searcher to enter the system appears at time $\lambda^{-1}(\sigma_1 + \cdots + \sigma_n)$. This sum follows the Erlang$(n, \lambda)$ distribution, which is a special case of the Gamma distribution. The Erlang has the survival probability
\begin{equation}
    \P(\Erlang(n, \lambda) > t) = \frac{\Gamma(n, \lambda t)}{(n-1)!}, \qquad \Gamma(s, x) = \int_x^\infty z^{s-1}e^{-z}dz,
    \label{eqn:erlangSurv_and_upIncGamma}
\end{equation}
where $\Gamma(s, x)$ is the upper incomplete gamma function.

Each searcher starts at the same initial position (or from a position drawn independently from the same probability distribution of initial positions) and subsequently moves iid. Rather than specifying the initial position and movement dynamics (and domain and target geometry, etc.), we work in terms of the FPT $\tau$ of a single searcher. In particular, if $\{\tau_n\}_{n\ge1}$ is an iid sequence of realizations of $\tau$, then the random variable $\tau_n$ is the length of time between when the $n$th searcher enters the system and when that same searcher finds the target. We stress that all of the information about the search process (initial searcher position, spatial domain, searcher motion, target location(s), etc.)\ is encoded in the random variable $\tau$ (i.e.\ in the probability distribution of $\tau$). The FPT to the target for the entire system is then the random variable $T_1$ defined in \eqref{eq:T10}. More generally, the $k$th FPT $T_k$ is the $k$th minimum of the set on the right-hand side of \eqref{eq:T10}. That is,
\begin{align*}
    T_k
    =\min_n\Big\{\Big\{\tau_n+\frac{1}{\lambda}\sum_{i=1}^n\sigma_i\Big\}\Big\backslash\cup_{j=1}^{k-1}T_j\Big\},\quad k\ge1.
\end{align*}

Throughout this paper, we assume that 
\begin{align}\label{eq:assump1}
    \tau = \begin{cases}
        t_{\min} &\withProb q>0,\\
        t_{\min}(1 + \widetilde{X}) &\withProb 1-q>0,
    \end{cases}
\end{align}
where $q\in(0,1)$, $\tmin>0$, and $\widetilde{X}>0$ is a strictly positive random variable whose short-time distribution satisfies either
\begin{align}\label{eq:assump2}
\begin{split}
    \P(\widetilde{X}<\eps)
    &\sim\alpha\eps^p\quad\text{as }\eps\to0,\\
    \text{or}\quad\P(\widetilde{X}<\eps)
    &\sim\alpha\ln(1/\eps)\eps^p\quad\text{as }\eps\to0,
\end{split}
\end{align}
where $\alpha>0$ and $p>0$. Throughout this paper, ``$f\sim g$'' denotes $f/g\to1$. 

We make several comments about the assumptions in \eqref{eq:assump1}-\eqref{eq:assump2}. First, \eqref{eq:assump1} implies that $\P(\tau<\tmin)=0$ for $\tmin>0$, which is the bounded speed assumption (i.e.\ searchers cannot reach the target in an arbitrarily short time). Second, we have assumed that $q:=\P(\tau=\tmin)>0$, which means that there is a strictly positive probability that $\tau=\tmin$. The case $q=0$ can be handled using the methods of Refs.~\cite{tung2025first, tung2025passage} after shifting time by $\tmin$ (see the \nameref{sec:discussion}). We also assume $q<1$ for most of the paper since $q=1$ is much easier and handled in Section \ref{sec:leading_order}. Third, \eqref{eq:assump1}-\eqref{eq:assump2} involve only the short time behavior because we are investigating the fast immigration limit; many searchers implies the target will be found quickly, which implies only behavior near time $t=t_{\min}$ is relevant. Fourth, as we show in section~\ref{sec:examples}, the assumptions in \eqref{eq:assump1}-\eqref{eq:assump2} cover a broad range of stochastic search processes. 
 
\section{\label{sec:math}Theoretical Results}
We start with an exact expression for the survival probability of the $k$th searcher in Section \ref{sec:exact_formula}. Using this, we obtain the leading order behavior of $T_k$ given in Section \ref{sec:leading_order}. We then give second order approximations of the distribution and moments of $T_k$ in Section \ref{sec:second_order}. 

To simplify the analysis in this section, we shift time by $t_{\min}$ and then rescale time by $t_{\min}$. This nondimensionalization makes the process have constant immigration rate $\widetilde{\lambda}$, searcher behavior $\widetilde{\tau}$, and $k$th fastest searcher FPT $\widetilde{T_k}$ related to the dimensional problem by
\begin{equation}
    \widetilde{\lambda} := \lambda t_{\min}, \quad \widetilde{\tau} := \frac{\tau-t_{\min}}{t_{\min}}, \quad \widetilde{T_k} := \frac{T_k-t_{\min}}{t_{\min}}.
    \label{eqn:rescaled}
\end{equation}
It follows that $\widetilde{\tau}$ is
\begin{equation}
    \widetilde{\tau} = \begin{cases}
        0 &\withProb q,\\
        \widetilde{X}, &\withProb 1-q,
    \end{cases} \quad 
    \P(\widetilde{X}<\eps) \sim \begin{cases}
        \alpha\eps^p \text{ or}, \\
        \alpha \ln(1/\eps)\eps^p.
    \end{cases}
    \label{eqn:tau_adj}
\end{equation}

\subsection{An exact survival probability}
\label{sec:exact_formula}
The rate of searchers immigrating into the system at any nonnegative time $t-s$ is $\wtlambda$. Therefore, the rate of searchers immigrating into the system at time $t-s$ that find the target before time $t$ is the product of the rate of immigration $\wtlambda$ and the probability $\P(\wttau < s)$ that a searcher finds the target within time $s$. As such, by properties of inhomogeneous Poisson point processes, the number of searchers that have found the target before time $t$ is Poisson distributed with mean
$$
\mu(t) := \int_0^t \wtlambda \P(\wttau <  s)ds.
$$
Using \eqref{eqn:tau_adj}, it follows that
\begin{equation}
    \mu(t) = \wtlambda q t + \wtlambda (1-q) \int_0^t \P(\wtX <s) ds.
    \label{eqn:mu}
\end{equation}
Since $\P(\widetilde{T_k} > t)$ is the probability that fewer than $k$ searchers have found the target by time $t$, it follows that
\begin{equation}
\P(\widetilde{T_k} > t) = e^{-\mu(t)}\sum_{i=0}^{k-1}\frac{\mu(t)^i}{i!} = \frac{\Gamma(k, \mu(t))}{(k-1)!}.
\label{eqn:exact_formula}
\end{equation}

\subsection{Leading order behavior}
\label{sec:leading_order}
We give results for convergence in distribution and convergence of moments. We leave the proofs for the Appendix.

For convergence in distribution, \eqref{eqn:tau_adj} implies that $\P(\wttau < t) \ll 1$ for small $t$, and therefore $\mu(t) \approx \wtlambda q t$ for small $t$. This is the mean for a Poisson point process with constant rate $\wtlambda$, which implies $\widetilde{T_k}$ is roughly Erlang$(k, \wtlambda q)$ distributed.  
\begin{theorem}
    $$
    \widetilde{\lambda} q\widetilde{T_k}\rightarrow \Erlang(k, 1)
    $$
    in distribution as $\wtlambda \rightarrow \infty$. Equivalently,
    $$
    \lim_{\wtlambda \rightarrow \infty}\P\left(\wtlambda q\widetilde{T_k} >  x\right) = \frac{\Gamma(k, x)}{(k-1)!}.
    $$
    \label{thm:converge_distribution}
\end{theorem}

One might also expect the moments of $\widetilde{\lambda} q\widetilde{T_k} $ to converge to the moments of an Erlang distribution, and indeed,
\begin{theorem}
    For any nonnegative integer $m$,
    $$
    \lim_{\wtlambda \rightarrow \infty} \E\left[\left(\widetilde{\lambda} q\widetilde{T_k}\right)^m\right] = \frac{(k+m-1)!}{(k-1)!}.
    $$
    \label{thm:converge_moments}
\end{theorem}

Note that when $q=1$, $\widetilde{T_k}$ is not just roughly but exactly Erlang$(k, \wtlambda q)$ distributed. Therefore the limits in Theorems \ref{thm:converge_distribution} and \ref{thm:converge_moments} become unnecessary and
$$
\P\left(\wtlambda \widetilde{T_k} >  x\right) = \frac{\Gamma(k, x)}{(k-1)!},\quad \E\left[\left(\widetilde{\lambda} \widetilde{T_k}\right)^m\right] = \frac{(k+m-1)!}{(k-1)!}, \quad \when q=1.
$$

\subsection{Second order terms}
\label{sec:second_order}
The leading order behavior does not depend on $\wtX$. In this section, we incorporate the small time behavior of $\wtX$ to find higher order terms. We first note and justify in the Appendix that
\begin{equation}
    I(t):=\int_0^t \P(\wtX < s) ds \sim \begin{cases}
        \alpha \frac{t^{p+1}}{p+1} &\when \P(\wtX<\eps)\sim \alpha \eps^p \\
        \alpha \frac{\ln(1/t)}{p+1}t^{p+1}  &\when \P(\wtX<\eps)\sim \alpha \ln(1/\eps)\eps^p
    \end{cases}
    \label{eqn:integral_approx}
\end{equation}

For approximating the distribution, it follows from \eqref{eqn:exact_formula} that 
\begin{align*}
    \P\left(\widetilde{T_k} > \frac{x}{\wtlambda}\right) &= \frac{\Gamma(k, \mu(x/\wtlambda))}{(k-1)!}  \\
    &= \frac{1}{(k-1)!} \Gamma\left(k, q x + \wtlambda (1-q) I(x/\wtlambda)\right) \\
    &\approx \frac{\Gamma\left(k, q x\right)}{(k-1)!}  - \frac{(qx)^{k-1}e^{-qx}}{(k-1)!} \wtlambda (1-q) I(x/\wtlambda),
\end{align*}
where the last line results from a linear approximation of the upper incomplete gamma function. Noting that $x/\wtlambda$ is small for sufficiently large $\wtlambda$, one can plug in \eqref{eqn:integral_approx} for a second term. In particular, in terms of dimensional variables, we have the following approximation if $\P(\wtX<t)\sim \alpha t^p$,
\begin{align*}
    P\Big(T_k>\tmin+\frac{y}{q\lambda}\Big)
    \approx\frac{\Gamma(k,y)}{\Gamma(k)}
    -\frac{y^{k}e^{-y}}{\Gamma(k)}\Big(\frac{y}{q\lambda\tmin}\Big)^p\frac{(1-q)\alpha}{q(1+p)}.
\end{align*}
Similarly, if $\P(\wtX<t)\sim \alpha \ln(1/t) t^p$, then
\begin{align*}
    &P\Big(T_k>\tmin+\frac{y}{q\lambda}\Big)
    \approx\frac{\Gamma(k,y)}{\Gamma(k)}-\frac{y^{k}e^{-y}}{\Gamma(k)}\Big(\frac{y}{q\lambda\tmin}\Big)^p\frac{(1-q)\alpha}{q(1+p)}\ln\Big(\frac{q\lambda\tmin}{y}\Big).
\end{align*}

For approximating the moments, we outline here a heuristic derivation of the second term and leave the rigorous derivation in the Appendix. It is well known from the theory of inhomogeneous Poisson point process that we can generate $\{\widetilde{T}_i\}_i$ by taking the points $\{\mathcal{T}_i\}_i$ from a rate $1$ Poisson point process and solving
\begin{equation}
    \mcTk = \mu(\wtTk).
    \label{eqn:transformation}
\end{equation}
When $t \ll \wtlambda$, then 
$$
\mu^{-1}(t) \approx \frac{t}{\wtlambda q} - \frac{1-q}{q}I(t/(\wtlambda q)).
$$
As such,
\begin{align*}
    \E\left[\wtTk^m\right] &= \E\left[(\mu^{-1}(\mcTk))^m\right] \\
    &\approx \E\left[\left(\frac{\mcTk}{\wtlambda q} - \frac{1-q}{q}I(\mcTk/(\wtlambda q))\right)^m\right] \\
    &\approx \E\left[\left(\frac{\mcTk}{\wtlambda q}\right)^m - m\frac{1-q}{q}I(\mcTk/(\wtlambda q))\left(\frac{\mcTk}{\wtlambda q}\right)^{m-1}\right].
\end{align*}
Using \eqref{eqn:integral_approx} and
\begin{equation}
    \E[(\mathcal{T}_k)^n] = \frac{\Gamma\left(k+n\right)}{(k-1)!}, \quad \E[\mathcal{T}_k^n\ln(1/\mcTk)] = -\frac{\Gamma(k+n)}{(k-1)!} \psi(k+n),
    \label{eqn:integrate_mcTk}
\end{equation}
where $n>0$ and $\psi$ is the digamma function leads to
\begin{equation}
    \E\left[(\wtlambda q\wtTk)^m\right] \approx \frac{\Gamma\left(k+m\right)}{(k-1)!} - m\frac{1-q}{q}H(m, \wtlambda q, \alpha, p),
    \label{eqn:two_term_moment}
\end{equation}
where
\begin{equation*}
    H(m, \wtlambda q, \alpha, p) = \frac{\alpha \Gamma\left(k+p+m\right)}{(\wtlambda q)^{p}(k-1)!(p+1)}\begin{cases}
        1 &\when \P(\wtX<t)\sim \alpha \eps^p, \\
        \ln(\wtlambda q)  &\when \P(\wtX<t)\sim \alpha \ln(1/\eps)\eps^p.
    \end{cases}
\end{equation*}

In terms of the dimensional variables, if $\P(\wtX<t)\sim \alpha t^p$, then the approximation in \eqref{eqn:two_term_moment} is
\begin{align*}
    \E[(T_k-\tmin)^m]
    \approx\frac{1}{(q\lambda)^m}\frac{\Gamma(k+m)}{\Gamma(k)}
    -\frac{1}{(q\lambda)^{m+p}}\frac{1}{(\tmin)^p}\frac{m(1-q)\alpha\Gamma(k+p+m)}{q\Gamma(k)(1+p)}.
\end{align*}
Similarly, if $\P(\wtX<t)\sim \alpha \ln(1/t)t^p$, then the approximation in \eqref{eqn:two_term_moment} is
\begin{align*}
    \E[(T_k-\tmin)^m]
    &\approx\frac{1}{(q\lambda)^m}\frac{\Gamma(k+m)}{\Gamma(k)}\\
    &\quad-\frac{1}{(q\lambda)^{m+p}}\frac{1}{(\tmin)^p}\frac{m(1-q)\alpha\Gamma(k+p+m)}{q\Gamma(k)(1+p)}\ln(q\lambda\tmin).
\end{align*}

\section{\label{sec:examples}Examples}

We now illustrate the results of section~\ref{sec:math} in some prototypical examples of stochastic search. We also compare the predictions of the analysis to numerical simulations. The simulation methods are detailed in Appendix~\ref{sec:simulationdetails}.

\subsection{\label{sec:1Drtp}1D run-and-tumble}

Suppose searchers move by a 1D run-and-tumble process with constant speed $v>0$ and tumbling rate $r>0$. In particular, the stochastic position $X(t)$ of a searcher after time $t\ge0$ has elapsed since entering the system satisfies
\begin{align*}
    \frac{\dd}{\dd t}X(t)
    =\begin{cases}
        +v>0 & \text{if }J(t)=+1,\\
        -v<0 & \text{if }J(t)=-1,
    \end{cases}
\end{align*}
where $J(t)\in\{-1,1\}$ is a continuous-time, two-state Markov chain that jumps at rate $r>0$,
\begin{align*}
    -1\Markov{r}{r}1.
\end{align*}
Suppose that searchers enter the system at $X(0)=0$ and suppose that there are targets at $x=\pm L$. The FPT of a searcher $\tau$ after entering the system is thus
\begin{align}\label{eq:tauexample}
    \tau
    :=\inf\{t\ge0:|X(t)|\ge L\}.
\end{align}
It is immediate that
\begin{align}\label{eq:vals1}
    \P(\tau<\tmin)
    &=0,\quad
    \P(\tau=\tmin)
    =q:=e^{-r\tmin}
    =e^{-\kappa},
\end{align}
where
\begin{align}\label{eq:vals2}
    \tmin:=L/v
    ,\quad \kappa
    :=r L/v.
\end{align}
Furthermore, it was shown in section~4.2 of Ref.~\cite{lawley2021pdmp} that 
\begin{align}\label{eq:vals3}
    \P(\tmin<\tau<\tmin(1+\delta))
    \sim(1-q)\alpha\delta^p,\quad\text{as }\delta\to0^+,
\end{align}
where
\begin{align*}
    p
    =1,\quad
    \alpha
    =\frac{\kappa}{2}\Big(\frac{1+\kappa}{e^\kappa-1}\Big).
\end{align*}

\begin{figure}
	\centering
	\includegraphics[width=1\textwidth]{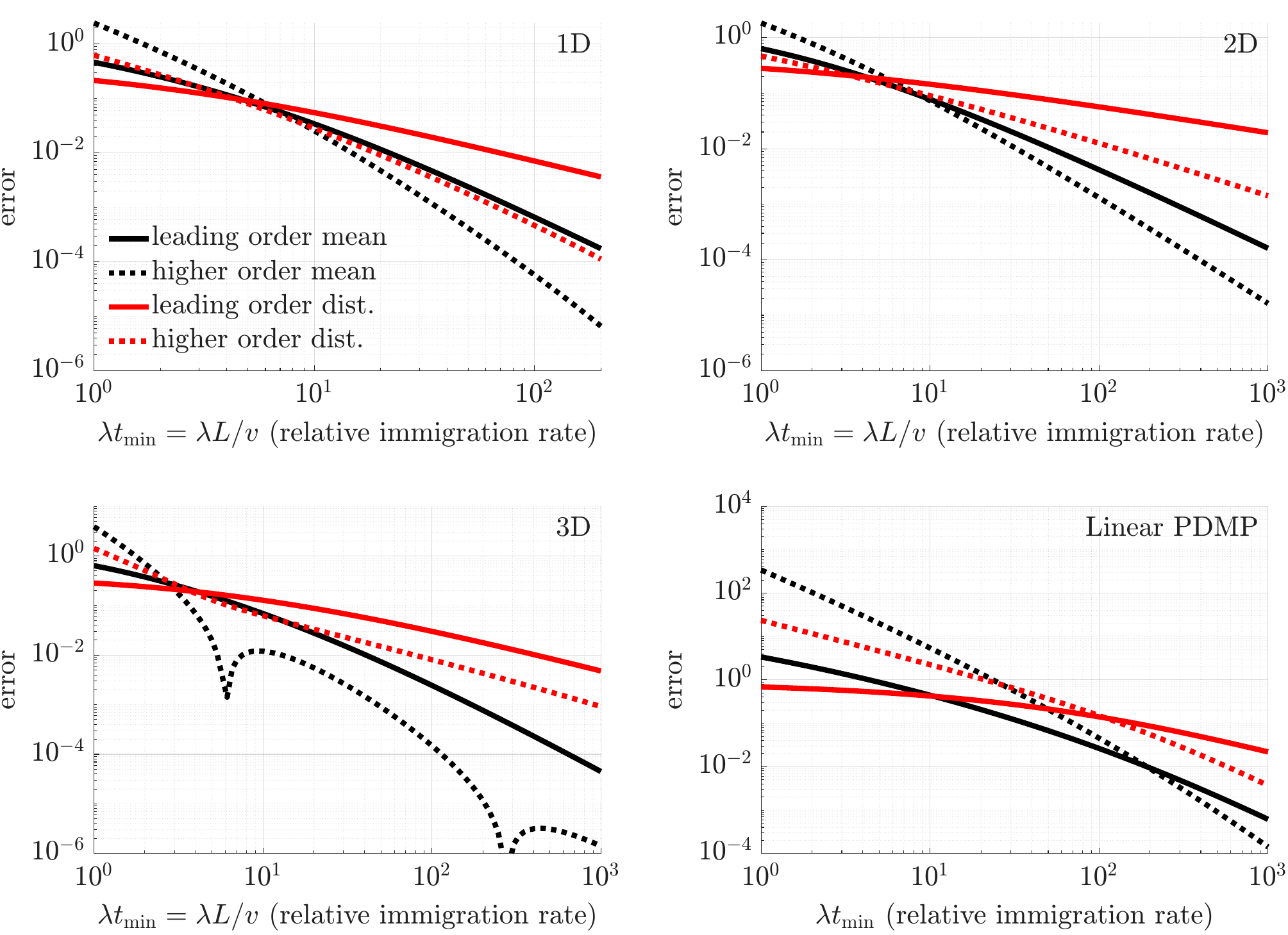}
    \caption{Comparison of theory of section~\ref{sec:math} with stochastic simulations of four search models. The black curves plot the relative error for approximations to the mean {fFPT}. The red curves plot the error for the approximations to the full distribution of the {fFPT}, where the error is the Kolmogorov-Smirnov distance (i.e.\ the maximum absolute difference between the cumulative distribution functions).}
	\label{fig:sims}
\end{figure}

Therefore, we can apply the theory of section~\ref{sec:math} to approximate the moments and distribution of the $k$th {fFPT} $T_k$ in the fast immigration limit ($\lambda\to\infty$). For instance, the mean {fFPT} is approximately
\begin{align}
    \E[T_1]
    &\approx\tmin
    +\frac{1}{q\lambda}
    -\frac{1}{(q\lambda)^2}\frac{1}{\tmin}\frac{\alpha(1-q)}{q}\nonumber\\
    &=\tmin
    +\frac{e^\kappa}{\lambda}
    -\Big(\frac{e^\kappa}{\lambda}\Big)^2r(1+\kappa)/2.\label{eq:1dapprox}    
\end{align}
Similarly, the distribution of the {fFPT} is approximately
\begin{align}
    P\Big(T_1>\tmin+\frac{x}{q\lambda}\Big)
    &\approx
    e^{-x}\Big[1-x^2\frac{(1-q)\alpha}{2q^2\lambda\tmin}\Big]\nonumber\\
    &=
    e^{-x}\Big[1-x^2\frac{e^{\kappa } (1+\kappa) r}{4 \lambda }\Big].\label{eq:1dapproxDist} 
\end{align}

The top left panel of Figure~\ref{fig:sims} plots the relative error in the approximation \eqref{eq:1dapprox}, where the solid black curve is for the first two terms in \eqref{eq:1dapprox} and the dashed black curve is for all three terms in \eqref{eq:1dapprox}. The red curves plots the error in the approximation \eqref{eq:1dapproxDist}, where the solid curve is for the first term in \eqref{eq:1dapproxDist} and the dashed curve is for both terms in \eqref{eq:1dapproxDist}. For the red curves, the error in distribution plotted on the vertical axis is the Kolmogorov-Smirnov distance (i.e.\ the maximum absolute difference between the cumulative distribution functions).

\subsection{\label{sec:2Drtp}2D run-and-tumble}

Suppose searchers move by a 2D analog of the run-and-tumble process considered in section~\ref{sec:1Drtp} above. Specifically, after entering the system, searchers start at the origin of the plane $X(0)=(0,0)\in\R^2$ and move at constant speed $v>0$ in a uniformly chosen random direction. The searcher moves in this initial direction until an exponentially distributed time with mean $1/r$ has elapsed, at which point the searcher chooses a new independent and uniform random direction. This process continues indefinitely. 

Suppose the target is the perimeter of a circle with radius $L$ centered around the origin so that the FPT of a searcher $\tau$ after entering the system is in \eqref{eq:tauexample} (where $|X(t)|$  is now interpreted as the radius of the searcher position at time $t$). It is immediate that \eqref{eq:vals1}-\eqref{eq:vals2} hold. Furthermore, it was shown in section~5 of Ref.~\cite{lawley2021pdmp} that \eqref{eq:vals3} holds with
\begin{align*}
    p
    =1/2,\quad
    \alpha
    =\frac{\kappa\sqrt{2}}{e^{\kappa}-1}.
\end{align*}
We can therefore immediately apply the theory of section~\ref{sec:math} to estimate the moments and distribution of the $k$th {fFPT}. The top right panel of Figure~\ref{fig:sims} plots the relative error in the approximations to the mean and distribution of the {fFPT} $T_1$ (analogously to the top left panel).

\subsection{3D run-and-tumble}

Suppose searchers move by the 3D analog of the run-and-tumble process considered in section~\ref{sec:2Drtp} above. Suppose the target is the entire radius at $L$ so that the FPT $\tau$ after entering the system is in \eqref{eq:tauexample} (where $|X(t)|$  is now interpreted as the radius of the searcher position at time $t$). It is immediate that \eqref{eq:vals1}-\eqref{eq:vals2} hold. Furthermore, it was shown in section~6 of Ref.~\cite{lawley2021pdmp} that
\begin{align*}
    \P(\tmin<\tau<\tmin(1+\delta))
    \sim(1-q)\alpha\ln(1/\delta)\delta,\quad\text{as }\delta\to0^+,
\end{align*}
where
\begin{align*}
    \alpha
    =\frac{\kappa}{e^{\kappa}-1}.
\end{align*}
We can therefore immediately apply the theory of section~\ref{sec:math} to estimate the moments and distribution of the $k$th {fFPT}. The bottom left panel of Figure~\ref{fig:sims} plots the relative error in the approximations to the mean and distribution of the {fFPT} $T_1$ (analogously to the top left panel). 

\subsection{Linear PDMP}

Suppose that $X$ moves according to the following linear PDMP,
\begin{align}\label{eq:linearPDMP}
    \frac{\dd}{\dd t}X(t)
    =\begin{cases}
        -X(t) & \text{if }J(t)=-1,\\
        1-X(t) & \text{if }J(t)=1,
    \end{cases}
\end{align}
where $J(t)\in\{-1,1\}$ is a two-state continuous-time Markov chain that jumps at rate $r>0$. The PDMP \eqref{eq:linearPDMP} has been used to model gene expression \cite{smiley2010gene} and storage systems \cite{boxma2005on}. Suppose $X(0)=1$ and consider the FPT to some threshold $\theta\in(0,1)$,
\begin{align*}
    \tau
    :=\inf\{t\ge0:X(t)\le \theta\}.
\end{align*}
If $p_0:=\P(J(0)=-1)>0$, then 
\begin{align*}
    \tmin
    :=\ln(1/\theta),\quad
    q:=p_0e^{-r\tmin}.
\end{align*}
Furthermore, it was shown in section~7 of Ref.~\cite{lawley2021pdmp} that \eqref{eq:vals3} holds with
\begin{align*}
    p
    =1,\quad
    \alpha
    =\frac{p_0r^2\theta^r(1-\theta)\ln(1/\theta)+(1-p_0)r\theta^r\ln(1/\theta)}{1-q}.
\end{align*}
We can therefore immediately apply the theory of section~\ref{sec:math} to estimate the moments and distribution of the $k$th {fFPT}. The bottom right panel of Figure~\ref{fig:sims} plots the relative error in the approximations to the mean and distribution of the {fFPT} $T_1$ (analogously to the top left panel). 

\section{\label{sec:diffusion}Diffusion and fast immigration limits do not commute}

We now consider three different models of diffusion and investigate their {fFPT}s in the fast immigration limit. The ``diffusion'' models are Brownian motion, a random walk on a discrete state space, and a run-and-tumble process. We find that the {fFPT}s of these processes can be very similar across a broad range of parameters, though the similarity eventually breaks down for sufficiently fast immigration. In this sense, the ``diffusion limit'' and the fast immigration limit do not commute.

\subsection{\label{sec:diff1}``Diffusion'' as Brownian motion}

We start with searchers which move by Brownian motion. Precisely, suppose that the position of a searcher after time $t\ge0$ has elapsed since entering the system is
\begin{align*}
    X_{\bm}(t)
    =\sqrt{2D}\,W(t)\quad t\ge0,
\end{align*}
where $D>0$ is the diffusivity and $W=\{W(t)\}_{t\ge0}$ is a standard Brownian motion. Hence, $X_{\bm}(0)=0$ and $\E[(X_{\bm}(t))^2]=2Dt$ for all $t\ge0$.

Let $\tau_{\bm}$ be the FPT of $X_{\bm}$ to $L>0$,
\begin{align*}
    \tau_{\bm}
    :=\inf\{t\ge0:X_{\bm}(t)>L\}.
\end{align*}
The survival probability of $\tau_{\bm}$ can be written in terms of the error function \cite{carslaw1959},
\begin{align}\label{eq:Sbm}
    S_{\bm}(t)
    :=\P(\tau_{\bm}>t)
    =\mathrm{erf}\bigg(\sqrt{\frac{L^2}{4Dt}}\bigg).
\end{align}
Let $T_{\bm}$ denote the corresponding {fFPT} (i.e.\ $T_{\bm}$ is defined as in \eqref{eq:T10}). Throughout this section, we omit the subscript ``1'' on {fFPT}s to simplify notation. The probability distribution and all the moments of $T_{\bm}$ in the fast immigration limit were obtained in Ref.~\cite{tung2025first}. The mean {fFPT} has the following asymptotic expansion \cite{tung2025first},
\begin{align}
    \E[T_{\bm}]
    &=\frac{C}{(p+2)W}-\frac{\gamma C}{(p+2)^2W(W+1)}+\cdots\label{eq:lambertw}\\
    &\sim\frac{L^2}{4D\ln(\lambda L^2/(4D))}\quad\text{as }\lambda L^2/D\to\infty,\label{eq:leading}
\end{align}
where $\gamma\approx 0.5772$ is the Euler-Mascheroni constant, $C=L^2/(4D)$, $p=1/2$, $A=\sqrt{4D/(L^2\pi)}$, and
\begin{align*}
    W
    =W_0\Big(\frac{1}{p+2}\big(AC^{p+1}\lambda\big)^{\frac{1}{p+2}}\Big),
\end{align*}
where $W_0(z)$ denotes the principal branch of the LambertW function \cite{corless1996}. The leading order asymptotic in \eqref{eq:leading} was first obtained by Campos and M\'endez \cite{campos2024dynamic} and follows from \eqref{eq:lambertw} and the asymptotics of the LambertW function \cite{corless1996}. 

\subsection{\label{sec:diff2}``Diffusion'' on a discrete state space}

We now consider a discrete space analog of the Brownian diffusion process $X_{\bm}$. Discretize space with step size
\begin{align}\label{dx}
\Delta x
=\eps L
\ll L,
\end{align}
where $0<\eps\ll1$. 
Let $\{X_{\disc}(t)\}_{t\ge0}$ be the continuous-time Markov chain on the discrete lattice,
\begin{align}\label{1d}
I
:=\{i\Delta x\}_{i\in\mathbb{Z}}
=\{\dots,-\Delta x,0,\Delta x,\dots\},
\end{align}
which jumps to neighboring lattice points according to
\begin{align*}
    X_{\disc}
    &\to X_{\disc}\pm\Delta x\quad\text{at rate }r:=\frac{D}{(\Delta x)^{2}}
    =\frac{D}{\eps^2 L^{2}}>0.
\end{align*}
Put another way, $X_{\disc}$ spends an exponentially distributed amount of time at its current position with mean $\eps^2 L^2/(2D)$ and then moves $\Delta x$ to the right with probability $1/2$ and $\Delta x$ to the left with probability $1/2$.
Assume $X_{\disc}(0)=0$. It is well-known that $X_{\disc}$ converges to $X_{\bm}$ as $\eps\to0$ \cite{billingsley2013}. The similarity between $X_{\disc}$ and $X_{\bm}$ for $0<\eps\ll1$ is perhaps most easily seen by noticing that the forward Fokker-Planck equation (i.e.\ the master equation) governing the probability density of $X_{\disc}$ is identical to a centered, second order finite difference approximation for the spatial derivative in the Fokker-Planck equation which governs the spatiotemporal of the probability density of $X_{\bm}$.

Define $\tau_{\disc}$ analogously to $\tau_{\bm}$,
\begin{align*}
\tau_{\disc}
&:=\inf\{t>0:X_{\disc}(t)>L\},
\end{align*}
and let $T_{\disc}$ denote the corresponding {fFPT} in \eqref{eq:T10}. 
The survival probability of $\tau_{\disc}$ is
\begin{align}\label{eq:Sdisc}
    S_{\disc}(t)
    :=\P(\tau_{\disc}>t)
    =e^{-2rt}\Big[I_0(2rt)+I_d(2rt)+2\sum_{k=1}^{d-1}I_k(2rt)\Big],
\end{align}
where $I_k$ denotes the modified Bessel function of the first kind and $d\ge1$ is the smallest integer greater than or equal to $L/\Delta x=1/\eps$,
\begin{align*}
    d
    =\lceil L/\Delta x\rceil
    =\lceil 1/\eps \rceil.
\end{align*}
One can verify that the formula for $S_{\disc}(t)$ in \eqref{eq:Sdisc} is correct by checking that it satisfies the appropriate backward Kolmogorov equations (see Appendix~\ref{sec:checkbackward}).

Proposition~1 in Ref.~\cite{lawley2020networks} implies that $\tau_{\disc}$ has the following short-time probability distribution (which can also be checked directly from \eqref{eq:Sdisc}),
\begin{align*}
    \P(\tau_{\disc}\le t)
    \sim \frac{r^d}{d!}t^d\quad\text{as }t\to0^+.
\end{align*}
Therefore, Theorems 3.1 and 3.4 in Ref.~\cite{tung2025first} give the probability distribution and moments of the {fFPT} $T_{\disc}$ in the fast immigration limit. In particular, the mean {fFPT} satisfies
\begin{align*}
    \E[T_{\disc}]
    \sim \frac{B(d)}{{{r}}}\Big(\frac{{{r}}}{\lambda}\Big)^{1/(d+1)}\quad\text{as }\lambda/r=\lambda \eps^2L^2/D\to\infty,
\end{align*}
where the prefactor is $B(d)=((d+1)!)^{1/(d+1)}\Gamma(1+1/(d+1))\approx (d+1)/e$.

\subsection{\label{sec:diff3}``Diffusion'' at bounded speed}

We now consider a bounded speed analog of the diffusion processes $X_{\bm}$ and $X_{\disc}$ considered above. Specifically, let $X_{\rtp}$ be the 1D run-and-tumble process considered in section~\ref{sec:1Drtp} whose constant speed $v$ and tumbling rate $r$ are given by
\begin{align*}
    v
    =\frac{v_0}{\eps},\quad
    r
    =\frac{v^2}{2D}
    =\frac{v_0^2}{2D\eps^2},
\end{align*}
where $0<\eps\ll1$. Assume that $\P(J(0)=1)=1$, which means that $X_{\rtp}$ is initially increasing. Let $\tau_{\rtp}$ be the FPT of $X_{\rtp}$ (analogous to $\tau_{\bm}$ and $\tau_{\disc}$),
\begin{align*}
    \tau_{\rtp}
    =\inf\{t\ge0:X_{\rtp}(t)>L\},
\end{align*}
and let $T_{\rtp}$ be the corresponding {fFPT}. The survival probability of $\tau_{\rtp}$ is \cite{malakar2018steady, grebenkov2026fastest}
\begin{align}\label{eq:Srtp}
    S_{\rtp}(t)
    :=\P(\tau_{\rtp}>t)
    =1-\Theta(t-\tmin)\bigg[e^{-\kappa}+\kappa\int\limits_1^{t/\tmin}\frac{I_1(\kappa\sqrt{z^2-1})}{e^{\kappa z}\sqrt{z^2-1}}\,\dd z\bigg],
\end{align}
where $\Theta(x)=1$ if $x\ge0$ and $\Theta(x)=0$ if $x<0$, $I_1$ is the modified Bessel function of the first kind, $\kappa=r L/v$, and $\tmin=L/v$.

It is well-known that if we fix $v_0>0$ and $D>0$ and take $\eps\to0$, then $X_{\rtp}$ converges to the Brownian diffusion process $X_{\bm}$ \cite{billingsley2013}. 
Nevertheless, the theory of section~\ref{sec:math} implies that if we fix $\eps>0$ and take the fast immigration limit, then the distribution and moments of the {fFPT} $T_{\rtp}$ are governed by the theory of section~\ref{sec:math}. For instance, the mean {fFPT} satisfies as $\lambda\to\infty$ (see section~\ref{sec:1Drtp}),
\begin{align*}
    \E[T_{\rtp}]
    -\tmin
    \sim \frac{e^\kappa}{\lambda}\quad\text{where }\kappa=\frac{rL}{v}=\frac{v_0L}{2D\eps}.
\end{align*}

\subsection{Comparison of three ``diffusion'' processes}

We now investigate the mean {fFPT} for the three ``diffusion'' processes introduced in sections~\ref{sec:diff1}-\ref{sec:diff3}. If $S(t)=\P(\tau>t)$ denotes the survival probability of a FPT $\tau$, then the survival probability of the {fFPT} $T_1$ is
\begin{align*}
    \P(T_1>t)
    =\exp\bigg[-\lambda\int_0^t (1-S(s))\,\dd s\bigg].
\end{align*}
Therefore, the mean {fFPT} can be written in terms of the survival probability of $\tau$,
\begin{align}\label{eq:intrep}
    \E[T_1]
    =\int_0^\infty \P(T_1>t)\,\dd t
    =\int_0^\infty \exp\bigg[-\lambda\int_0^t (1-S(s))\,\dd s\bigg]\,\dd t.
\end{align}

\begin{figure}
	\centering
	\includegraphics[width=1\textwidth]{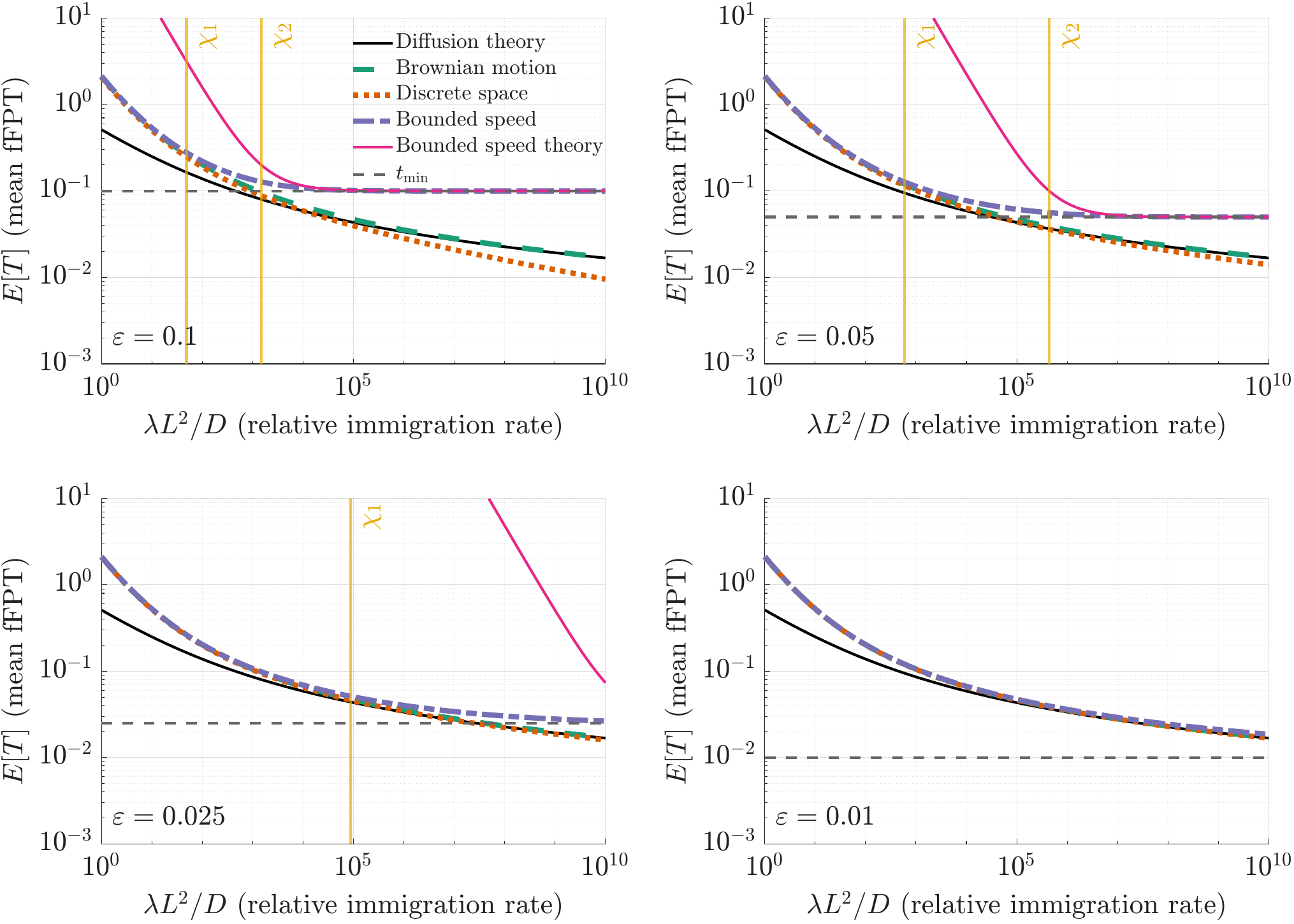}
    \caption{Comparison of the mean {fFPT}s for the three ``diffusion'' processes of sections~\ref{sec:diff1}-\ref{sec:diff3}.}
	\label{fig:diffusion}
\end{figure}

In Figure~\ref{fig:diffusion}, we plot the mean {fFPT} for the three ``diffusion'' processes by numerically computing the integral in \eqref{eq:intrep} for the survival probability $S$ given by either $S_{\bm}$ in \eqref{eq:Sbm}, $S_{\disc}$ in \eqref{eq:Sdisc}, or $S_{\rtp}$ in \eqref{eq:Srtp}. The solid black curve is the approximation in \eqref{eq:lambertw} and the dashed green curve is for the Brownian motion process of section~\ref{sec:diff1} (these two curves are the same in all four panels of Figure~\ref{fig:diffusion}). The dotted orange curve is for the discrete space process of section~\ref{sec:diff2}. The dash-dotted purple curve is the bounded speed process of section~\ref{sec:diff3}, and the dashed horizontal line is the corresponding minimal time $\tmin=L/v=\eps L/v_0$. The solid pink curve is the mean {fFPT} approximation of Theorem~\ref{thm:converge_moments}. The four panels are for different values of $\eps$. 

Notice from Figure~\ref{fig:diffusion} that the {fFPT}s of the three ``diffusion'' processes (and the approximation in \eqref{eq:lambertw}) closely agree across a broad range of immigration rates $\lambda$ if $\eps\ll1$. This agreement eventually breaks down for sufficiently large immigration rate $\lambda$, and then the theory of section~\ref{sec:math} (solid pink curve from Theorem~\ref{thm:converge_moments}) closely approximates $\E[T_{\rtp}]$.

To estimate the parameter regime when $T_{\bm}$ and $T_{\rtp}$ agree, observe that the lower bound $T_{\rtp}\ge \tmin=L/v$ means that the asymptotic in \eqref{eq:leading} can approximate $\E[T_{\rtp}]$ only if
\begin{align*}
    \frac{L^2}{4D\ln(\lambda L^2/(4D))}
    \ge\tmin=\frac{L}{v}
    =\frac{\eps L}{v_0},
\end{align*}
which, upon algebraic rearrangement, is equivalent to
\begin{align}\label{eq:nottoobig}
    \lambda L^2/D
    =2\lambda\tmin \kappa\le\chi_1:=4e^{\kappa/2},\quad \text{where }\kappa
    =\frac{rL}{v}
    =\frac{r_0 L}{v_0\eps}.
\end{align}
To estimate the parameter regime when the theory of section~\ref{sec:math} describes $T_{\rtp}$, we require that the leading order correction to the mean implied by Theorem~\ref{thm:converge_moments} be small, which means $q\lambda\tmin\gg1$, which is equivalent to (using that $q=e^{-\lambda\tmin}=e^{-\kappa}$ and $\tmin=L/v$)
\begin{align}\label{eq:verybig}
    \lambda L^2/D\gg \chi_2:=2\kappa e^\kappa.
\end{align}
The thresholds $\chi_1$ and $\chi_2$ defined in \eqref{eq:nottoobig} and \eqref{eq:verybig} are plotted as vertical lines in Figure~\ref{fig:diffusion}. 
Summarizing, this analysis predicts that if $\eps\ll1$, then (i) $T_{\rtp}$ agrees with the Brownian motion {fFPT} $T_{\bm}$ in the regime of \eqref{eq:nottoobig}, and (ii) $T_{\rtp}$ agrees with the theory of section~\ref{sec:math} in the regime of \eqref{eq:verybig}. Furthermore, in the case of (i), if we also have that $\lambda L^2/D\gg1$, then both $T_{\rtp}$ and $T_{\bm}$ agree with the Brownian motion theory of Ref.~\cite{tung2025first} (the mean of which is given in \eqref{eq:lambertw}-\eqref{eq:leading}).

\section{\label{sec:discussion}Discussion}

In this paper, we investigated the {fFPT}s of stochastic search processes with immigration. We assumed that searchers cannot move arbitrarily fast, which is of course natural for physical applications. Our theory focused on the parameter regime of fast immigration, which is the regime where the bounded speed assumption affects {fFPT}s.

The bounded speed assumption manifested in our analysis as the assumption that the time $\tau$ that it takes a single searcher to find the target after entering the system cannot be smaller than a strictly positive time $\tmin>0$,
\begin{align*}
    \P(\tau<\tmin)=0.
\end{align*}
We assumed that there is a strictly positive probability that $\tau$ attains this minimal time,
\begin{align}\label{eq:qpositive}
    q
    :=\P(\tau=\tmin)>0.
\end{align}
We assumed \eqref{eq:qpositive} because \eqref{eq:qpositive} holds for search by run-and-tumble motion (and for search by piecewise deterministic Markov processes more generally \cite{lawley2021pdmp}).

Furthermore, the case $q=0$ can be handled using the methods of Refs.~\cite{tung2025first, tung2025passage} if we shift time by $\tmin$. To elaborate and compare the $q>0$ versus $q=0$ cases, suppose
\begin{align*}
    \tau = \begin{cases}
        t_{\min} &\withProb q\ge0,\\
        t_{\min}(1 + \widetilde{X}) &\withProb 1-q>0,
    \end{cases}
\end{align*}
where $q\in[0,1)$, $\tmin>0$, and $\widetilde{X}>0$ is a strictly positive random variable whose short-time distribution satisfies
\begin{align*}
    \P(\widetilde{X}<\eps)
    &\sim\alpha\eps^p\quad\text{as }\eps\to0,
\end{align*}
where $\alpha>0$ and $p>0$. If $q>0$, we found in section~\ref{sec:math} the following estimate of the mean of $T_k$ for fast immigration,
\begin{align*}
    \E[T_k]
    &\approx \tmin\bigg[1
    + \frac{k}{q\lambda\tmin}
    - \frac{1}{(q\lambda\tmin)^{p+1}}\frac{(1-q)\alpha\Gamma(k+p+1)}{q\Gamma(k)(p+1)}\bigg].
\end{align*}
If $q=0$, then it follows from the analysis in Refs.~\cite{tung2025first, tung2025passage} that in the fast immigration limit,
\begin{align*}
    \E[T_k]
    &\approx \tmin\bigg[1
    +\Gamma(k+\tfrac{1}{p+1})\Big(\frac{p+1}{\alpha\lambda\tmin}\Big)^{\frac{1}{p+1}}\bigg].
\end{align*}
Hence, if $q>0$, then the leading order correction to $\E[T_k]\approx\tmin$ has order $(q\lambda\tmin)^{-1}$, whereas if $q=0$, then the leading order correction to $\E[T_k]\approx\tmin$ has order $(q\lambda\tmin)^{-1/(p+1)}$.

This paper was motivated by how unbounded speed versus bounded speed models of diffusion must differ in the fast immigration limit (see \eqref{eq:fastimmigration0} and \eqref{eq:fastimmigration0discrete} versus \eqref{eq:bound0} in the \nameref{sec:introduction}). After developing a general mathematical theory (section~\ref{sec:math}) and illustrating it in some numerical examples (section~\ref{sec:examples}), we considered three canonical models of ``diffusion'' (two unbounded speed models and one bounded speed model) and investigated when their {fFPT}s agree and when they differ (section~\ref{sec:diffusion}). 

More broadly, the ubiquity of diffusion processes in nature can be understood as a consequence of the central limit theorem of probability theory. Indeed, the central limit theorem implies that the variegated details of different modes of stochastic motion yield identical Gaussian signatures. However, while the details of the motion do not affect typical stochastic behavior, such minutiae do affect stochastic extremes. The results of the present paper delineate when and how bounded versus unbounded search speeds affect {fFPT}s for search with immigration.

\subsubsection*{Acknowledgments}
SDL was supported by the National Science Foundation (Grant No.\ DMS-2325258). HRT was supported by a gift from William H. Miller III.


\bibliography{library.bib}
\bibliographystyle{abbrv}

\appendix
\section{Proofs}
Throughout the proofs, we use \eqref{eqn:transformation}, duplicated below for convenience. 
$$
\mcTk = \mu(\wtTk) = \wtlambda q \wtTk + \wtlambda (1-q) I(\wtTk)
$$
As $I(t)$ is nonnegative, it follows that
\begin{equation}
    \mcTk \geq \wtlambda q \wtTk
    \label{eqn:ordering}
\end{equation}

We start with leading order results. We first show
\begin{lemma}
    $$\wtlambda q \wtTk \rightarrow \mcTk \text{ in probability}$$
    as $\wtlambda \rightarrow \infty$.
    \label{lem:converge_in_prob}
\end{lemma}
\begin{proof}
    By \eqref{eqn:transformation}, \eqref{eqn:ordering}, and the monotonicity of $I(t)$
    \begin{align*}
        \mcTk - \wtlambda q \wtTk &= \wtlambda (1-q) I(\wtTk) \\
        &\leq \wtlambda (1-q) I(\mcTk/(\wtlambda q)) 
    \end{align*}
    Substituting $s = u/(\wtlambda q)$ in the integral $I$, this simplifies to 
    $$
    \mcTk - \wtlambda q \wtTk \leq \frac{1-q}{q}\int_0^{\mcTk} \P\left(\wtX < \frac{s}{\wtlambda q}\right) ds \leq \frac{1-q}{q}\mcTk \P\left(\wtX < \frac{\mcTk}{\wtlambda q}\right)
    $$
    
    For any $\eps >0$ and any $r>0$, there exists some $t_r$ such that $\P(\mcTk > t_r) < \eps/r$. As such
    \begin{align*}
        \P(\mcTk - \wtlambda q \wtTk > \eps) &\leq \P\left(\frac{1-q}{q}\mcTk \P\left(\wtX < \frac{\mcTk}{\wtlambda q}\right) > \eps\right) \\
        &= \P\left(\frac{1-q}{q}\mcTk \P\left(\wtX < \frac{\mcTk}{\wtlambda q}\right) > \eps\middle| \mcTk > t_r \right)\P(\mcTk > t_r)\\ 
        &\phantom{=}+ \P\left(\frac{1-q}{q}\mcTk \P\left(\wtX < \frac{\mcTk}{\wtlambda q}\right) > \eps\middle| \mcTk \leq t_r \right)\P(\mcTk \leq t_r) \\
        &\leq \P(\mcTk > t_r)+ \P\left(\frac{1-q}{q}\mcTk \P\left(\wtX < \frac{\mcTk}{\wtlambda q}\right) > \eps\middle| \mcTk \leq t_r \right) \\
        &\leq \frac{\eps}{r} + \P\left(\frac{1-q}{q}t_r \P\left(\wtX < \frac{t_r}{\wtlambda q}\right) > \eps \right)\\
        &< 2\eps/r \text{ for sufficiently large }\wtlambda 
    \end{align*}
    The last line uses that $\P\left(\wtX < s\right) \rightarrow 0$ as $s\rightarrow 0$, which follows from our assumed short time behavior of $\wtX$. Since $r$ can be made arbitrarily large, we have convergence in probability. 
    
\end{proof}

Our leading order results follow swiftly.
\begin{proof}[Proof of Theorem \ref{thm:converge_distribution}]
    Convergence in probability implies convergence in distribution, so the result follows by Lemma \ref{lem:converge_in_prob}.
\end{proof}

\begin{proof}[Proof of Theorem \ref{thm:converge_moments}]
    By the continuous mapping theorem for convergence in probability, for $m\geq 0$
    $$
    (\wtlambda q \wtTk)^m \rightarrow \mcTk^m \text{ in probability}
    $$
    By nonnegativity of $\wtTk$ and \eqref{eqn:ordering}, we have
    $$
    0\leq (\wtlambda q \wtTk)^m \leq \mcTk^m
    $$
    As such, applying the dominated convergence theorem with convergence in probability and using the moments of the Erlang$(k, 1)$ distribution, we find
    $$
    \lim_{\wtlambda \rightarrow \infty} \E\left[\left(\widetilde{\lambda} q\widetilde{T_k}\right)^m\right] = \E\left[\mcTk^m\right] = \frac{(k+m-1)!}{(k-1)!}
    $$
    where $m$ is a nonnegative integer, yielding our desired result.
\end{proof}

Next, we move onto second order terms. As we rigorously justified the second order term in the convergence in distribution, it remains to show \eqref{eqn:integral_approx} and the second order terms for the moments.
\begin{proof}[Proof of \eqref{eqn:integral_approx}]
    By the short time behavior of $\wtX$, for any $\eps > 0$, there exists a $\delta$ such that when $s<\delta$, 
    $$
    (1-\eps)f(s) \leq \P(\wtX < s) \leq (1+\eps)f(s)
    $$
    where $f(s)$ is the function in \eqref{eqn:tau_adj} (ie $\alpha s^p$ or $\alpha \ln(1/s)s^p$). As such, when $t < \delta$,
    $$
    \left|I(t) - \int_0^t f(s) ds\right| =\left| \int_0^t \P(\wtX < s) - f(s) ds\right| \leq \eps \int_0^t f(s) ds
    $$
    Dividing through by the integral of $f$ then yields
    $$
    I(t) \sim \int_0^t f(s) ds
    $$
    Integrating $f$ by hand or with computer software yields 
    $$
    \int_0^t f(s) ds = \begin{cases}
        \alpha \frac{t^{p+1}}{p+1} &\when \P(\wtX<\eps)\sim \alpha \eps^p \\
        \alpha \frac{1+(p+1)\ln(1/t)}{(p+1)^2}t^{p+1} \sim \alpha \frac{\ln(1/t)}{p+1}t^{p+1}  &\when \P(\wtX<\eps)\sim \alpha \ln(1/\eps)\eps^p
    \end{cases}
    $$
    which is the desired result.
\end{proof}

Before proving the higher order terms, we first prove a few algebraic results that will be useful

\begin{lemma}
    Let $E$ be the event where $\mcTk/(\wtlambda q)<\delta$ for some $\delta<1$. Then
    \begin{equation}
         \lim_{\wtlambda \rightarrow \infty}\frac{\E[\mcTk^m|E^c]\P(E^c)}{(\wtlambda q \delta)^{m+k-1} e^{-\wtlambda q \delta}} = \frac{1}{(k-1)!}
         \label{eqn:organize1}
    \end{equation}
    \begin{equation}
         \lim_{\wtlambda \rightarrow \infty}\frac{\E[\mcTk^m|E]\P(E) - \frac{(m+k-1)!}{(k-1)!}}{-\frac{(\wtlambda q \delta)^{m+k-1} e^{-\wtlambda q \delta}}{(k-1)!}} = 1
         \label{eqn:organize2}
    \end{equation}

    \begin{equation}
         \lim_{\wtlambda \rightarrow \infty}\frac{\E\left[\left(\frac{\mcTk}{\wtlambda q}\right)^{m-1}\int_0^{\mcTk/(\wtlambda q)} f(s) ds \middle| E \right]\P(E)}{H(m, \wtlambda q, \alpha, p)/(\wtlambda q)^m} = 1
         \label{eqn:organize4}
    \end{equation}
    \label{lem:organize}
\end{lemma}

\begin{proof}[Proof of Lemma \ref{lem:organize}]
    For \eqref{eqn:organize1}, it is a well known property of the upper incomplete Gamma function that
    \begin{equation}
        \lim_{x\rightarrow \infty}\frac{\Gamma(s, x)}{x^{s-1}e^{-x}} = 1
        \label{eqn:upper_inc_gamma_asymptotics}
    \end{equation}
    As such, using properties of the Erlang distribution,
    \begin{align*}
        \frac{1}{{(\wtlambda q \delta)^{m+k-1} e^{-\wtlambda q \delta}}}\E[\mcTk^m|E^c]\P(E^c) &= \frac{1}{{(\wtlambda q \delta)^{m+k-1} e^{-\wtlambda q \delta}}}\int_{\wtlambda q \delta}^\infty s^m \frac{s^{k-1}e^{-s}}{(k-1)!}ds \\
        &= \frac{1}{{(\wtlambda q \delta)^{m+k-1} e^{-\wtlambda q \delta}}}\frac{\Gamma(m+k, \wtlambda q \delta)}{(k-1)!}\\
        &\rightarrow \frac{1}{(k-1)!}
    \end{align*}
    as $\wtlambda \rightarrow \infty$.

    For \eqref{eqn:organize2}, note that
    $$
    \E[\mcTk^m|E]\P(E) + \E[\mcTk^m|E^c]\P(E^c) = \E[\mcTk^m] \quad \Rightarrow \quad \frac{\E[\mcTk^m|E]\P(E) -\E[\mcTk^m]}{-\E[\mcTk^m|E^c]\P(E^c)} = 1
    $$
    Looking up or computing the moments of an Erlang distribution for $\E[\mcTk^m]$ and using \eqref{eqn:organize1} yields the desired result.

    For \eqref{eqn:organize4}, by definition of the Erlang distribution,
    \begin{align*}
        \E&\left[\left(\frac{\mcTk}{\wtlambda q}\right)^{m-1}\int_0^{\mcTk/(\wtlambda q)} f(s) ds \middle| E \right]\P(E)\\ 
        &= \int_0^{\wtlambda q \delta}\frac{t^{k-1}e^{-t}}{(k-1)!}\left(\frac{t}{\wtlambda q}\right)^{m-1}\int_0^{t/(\wtlambda q)} f(s) ds dt\\
        &= (\wtlambda q)^k\int_0^{\delta}\frac{u^{k+m-2}e^{-\wtlambda q u}}{(k-1)!}\int_0^{u} f(s) ds du
    \end{align*}
    We now plug in both cases of short time behavior $f$.
    
    In the case $f(s) = \alpha t^p$, we simplify to
    $$
    \alpha(\wtlambda q)^k\int_0^{\delta}\frac{u^{k+m + p-1}e^{-\wtlambda q u}}{(p+1)(k-1)!} du
    $$
    Applying Watson's lemma yields
    $$
    \lim_{\wtlambda \rightarrow \infty}\frac{\E\left[\left(\frac{\mcTk}{\wtlambda q}\right)^{m-1}\int_0^{\mcTk/(\wtlambda q)} \alpha s^p ds \middle| E \right]\P(E)}{\alpha \frac{\Gamma(k+m+p)}{(\wtlambda q)^{m+p} (p+1) (k-1)!}} = 1
    $$
    
    In the case of $f(s) = \alpha \ln(1/t) t^p$, we simplify to
    $$
    \frac{\alpha (\wtlambda q)^k}{(k-1)!}\left[\int_0^{\delta}\frac{u^{k+m+p-1}e^{-\wtlambda q u}}{(p+1)^2} du + \int_0^{\delta}\frac{u^{k+m+p-1} \ln(1/u)e^{-\wtlambda q u}}{p+1} du \right]
    $$
    The behavior of the first term is immediate from the work where $f(s) = \alpha s^p$ and is
    $$
    \lim_{\wtlambda \rightarrow \infty}\frac{\frac{\alpha (\wtlambda q)^k}{(k-1)!}\int_0^{\delta}\frac{u^{k+m+p-1}e^{-\wtlambda q u}}{(p+1)^2} du}{\alpha \frac{\Gamma(k+m+p)}{(\wtlambda q)^{m+p} (p+1)^2 (k-1)!}} = 1
    $$
    For the second term, we note
    \begin{align*}
        \frac{\alpha (\wtlambda q)^k}{(k-1)!} &\int_0^{\delta}\frac{u^{k+m+p-1} \ln(1/u)e^{-\wtlambda q u}}{p+1} du \\
        &=-\frac{\alpha (\wtlambda q)^k}{(k-1)!(p+1)}\frac{d}{da}\left[\int_0^{\delta}u^{a}e^{-\wtlambda q u} du\right]\Bigg|_{a = k+m+p-1}
    \end{align*}
    Applying Watson's lemma now yields the asymptotic behavior
    \begin{align*}
        -\frac{\alpha (\wtlambda q)^k}{(k-1)!(p+1)}&\frac{d}{da}\left[\frac{\Gamma(a+1)}{(\wtlambda q)^{a+1}}\right]\Bigg|_{a = k+m+p-1}\\
        &= -\frac{\alpha (\wtlambda q)^k}{(k-1)!(p+1)}\frac{d}{da}\left[\frac{\Gamma(a+1)}{(\wtlambda q)^{a+1}}\right]\Bigg|_{a = k+m+p-1}\\
        &= \frac{\alpha \Gamma(k+m+p) }{(k-1)!(p+1)(\wtlambda q)^{m+p}}\left[\ln(\wtlambda q) - \psi(k+m+p)\right]
    \end{align*}
    where $\psi$ is the digamma function. As the $\psi$ term does not depend on $\wtlambda$, we can ignore it. As this second term is asymptotically larger than the first term, we find
    $$
    \lim_{\wtlambda \rightarrow \infty}\frac{\E\left[\left(\frac{\mcTk}{\wtlambda q}\right)^{m-1}\int_0^{\mcTk/(\wtlambda q)} \alpha \ln(1/s)s^p ds \middle| E \right]\P(E)}{\alpha \frac{\Gamma(k+m+p)}{(\wtlambda q)^{m+p} (p+1) (k-1)!}\ln(\wtlambda q)} = 1
    $$
\end{proof}

\begin{proof}[Proof of \eqref{eqn:two_term_moment}]
    For any $\eps > 0$ there exists a $\delta >0 $ such that for all $t<\delta$,
    $$
    \left|\frac{\P(\wtX <t)}{f(t)}  - 1\right| < \eps, \qquad\left|\frac{I(t)}{\int_0^t f(s) ds}  - 1\right| < \eps
    $$
    We also impose that $\delta<1$ and again define $E$ as the event where $\mcTk/(\wtlambda q)<\delta$. Note that by \eqref{eqn:ordering}, $\wtTk < \delta$ when conditioned on $E$. 

    The proof strategy is to note that
    \begin{equation}
        \E[\wtTk^m] = \E[\wtTk^m|E]\P(E) + \E[\wtTk^m|E^c]\P(E^c)
        \label{eqn:condition_by_E}
    \end{equation}
    and use this to show the leading terms of $\E[\wtTk^m]$ come from $\E[\wtTk^m|E]\P(E)$. To do so, note by \eqref{eqn:ordering} and \eqref{eqn:organize1} that 
    $$
    \frac{\E[\wtTk^m|E^c]\P(E^c)}{(\wtlambda q \delta)^{m+k-1}e^{-\wtlambda q \delta}} \leq \frac{\E[\mcTk^m|E^c]\P(E^c)}{(\wtlambda q \delta)^{m+k-1}e^{-\wtlambda q \delta}} \rightarrow \frac{1}{(k-1)!}
    $$
    as $\wtlambda \rightarrow \infty$. The remainder of the proof finds the two leading terms of $\E[\wtTk^m|E]\P(E)$ and shows they decay slower than $\wtlambda^{m+k-1}e^{-\wtlambda q \delta}$.

    By \eqref{eqn:transformation} and how $\delta$ is chosen, we find that conditioned on $E$,
    $$
    \mu_{-}(\wtTk) \leq \frac{\mcTk}{\wtlambda q} \leq \mu_{+}(\wtTk), \qquad \mu_{\pm}(t):= t + (1\pm\eps)\frac{1-q}{q}\int_0^{t} f(s) ds \text{ for }0\leq t<\delta
    $$
    Because $\delta$ was chosen to be less than $1$, then $f(s)$ is positive and therefore $\mu_{\pm}$ is monotonic. This means the inverse exists and
    \begin{equation}
        \mu_{+}^{-1}\left(\frac{\mcTk}{\wtlambda q}\right) \leq \wtTk \leq \mu_{-}^{-1}\left(\frac{\mcTk}{\wtlambda q}\right)
        \label{eqn:bounding_wtTk}
    \end{equation}
    Furthermore, it is straightforward to show
    \begin{align*}
        \mu_{+}^{-1}(t) &\geq t-(1+\eps)\frac{1-q}{q}\int_0^{t} f(s) ds \\
        \mu_{-}^{-1}(t) &\leq t-(1-\eps)\frac{1-q}{q}\int_0^{t} f(s) ds + (1-\eps)\frac{1-q}{q}\int_{t-(1-\eps)\frac{1-q}{q}\int_0^{t} f(u) du}^{t} f(s)ds\\
        &\leq t-(1-\eps)\frac{1-q}{q}\int_0^{t} f(s) ds + \left((1-\eps)\frac{1-q}{q}\right)^2f(t)\int_{0}^{t} f(s)ds
    \end{align*}
    Applying a binomial approximation yields the two terms expansion
    $$
    \mu_{\pm}^{-1}\left(\frac{\mcTk}{\wtlambda q}\right)^m = \left(\frac{\mcTk}{\wtlambda q}\right)^m -(1\pm\eps)m\frac{1-q}{q}\left(\frac{\mcTk}{\wtlambda q}\right)^{m-1}\int_0^{\mcTk/(\wtlambda q)} f(s) ds + \text{h.o.t.}
    $$
    Taking expectations on both sides, multiplying through by $\P(E)$, and shuffling terms yields
    $$
    \lim_{\wtlambda \rightarrow \infty}\frac{\E\left[\mu_{\pm}^{-1}\left(\frac{\mcTk}{\wtlambda q}\right)^m\middle| E \right]\P(E)  - \E\left[\left(\frac{\mcTk}{\wtlambda q}\right)^m\middle| E \right]\P(E)}{-m\frac{1-q}{q}\E\left[\left(\frac{\mcTk}{\wtlambda q}\right)^{m-1}\int_0^{\mcTk/(\wtlambda q)} f(s) ds\middle| E \right]\P(E)} = 1 \pm \eps
    $$
    Applying \eqref{eqn:organize2} and \eqref{eqn:organize4} and noting $\eps$ can be made arbitrarily small yields the desired result.
\end{proof}


\section{\label{sec:checkbackward}Derivation of $S_{\disc}$ in \eqref{eq:Sdisc}}

We now derive $S_{\disc}$ in \eqref{eq:Sdisc}. Let $X(t)$ be the continuous-time Markov chain defined in section~\ref{sec:diff2} and define the FPT to the origin $\tau_0=\inf\{t\ge0:X(t)=0\}$. Define $S_0(t)=0$ and
\begin{align}\label{eq:Sdformula}
    S_d(t)
    &=e^{-2rt}\Big[I_0(2rt)+I_d(2rt)+2\sum_{k=1}^{d-1}I_k(2rt)\Big],\quad d\ge1.
\end{align}
We claim that
\begin{align}\label{eq:claim}
    S_d(t)=\P(\tau_0>t\,|\,X(0)=d\Delta x)\quad \text{for all integers }d\ge0.
\end{align}
Note that $S_{\disc}$ in \eqref{eq:Sdisc} in the main text is the FPT distribution to hit $d\Delta x$ starting from the origin, which is equivalent to \eqref{eq:claim}, since \eqref{eq:claim} is the FPT distribution to hit the origin starting from $d\Delta x$ (the latter is more natural when allowing the starting position to vary, which we do here in the Appendix).

It is immediate that \eqref{eq:claim} holds for $d=0$ since $S_0(t)=0$. For $d\ge1$, we have the desired initial condition $S_d(0)=1$ since $I_0(0)=1$ and $I_k(0)=0$ for all $k\ge1$. To verify \eqref{eq:claim}, it remains to show that $S_d$ satisfies the backward Kolmogorov equations,
\begin{align}\label{eq:backwarddiscrete}
    \frac{\dd}{\dd t}S_d
    =r(S_{d+1}-2S_d+S_{d-1}),\quad \text{for all integers } d\ge1.
\end{align}
Setting $I_k=I_k(2rt)$ to simplify notation and taking the derivative of \eqref{eq:Sdformula} for $d\ge2$ gives
\begin{align*}
    \frac{\dd}{\dd t}S_d
    &=-2rS_d+re^{-2rt}\Big[2I_1+I_{d-1}+I_{d+1}+2I_0+2\sum_{k=1}^{d-2}I_k+2\sum_{k=2}^{d}I_k\Big]\\
    &=-2rS_d+re^{-2rt}\Big[I_0+I_{d-1}+2\sum_{k=1}^{(d-1)-1}I_k+I_0+I_{d+1}+2\sum_{k=1}^{(d+1)-1}I_k\Big]\\
    &=-2rS_d+rS_{d-1}+rS_{d+1},\quad\text{if }d\ge2,
\end{align*}
where we have used that $I_{-1}=I_1$ and the relation
\begin{align*}
    \frac{\dd}{\dd z}I_n(z)
    =\frac{1}{2}(I_{n-1}+I_{n+1}).
\end{align*}
For $d=1$, we have
\begin{align*}
    \frac{\dd}{\dd t}S_1
    =-2rS_1+re^{-2rt}(I_0+I_2+2I_1+0)
    =-2rS_1+r(S_2+S_0).
\end{align*}
Hence, we have verified \eqref{eq:backwarddiscrete}.

\section{\label{sec:simulationdetails}Simulation methods}

We now describe the stochastic simulation methods used in section~\ref{sec:examples}. First, we generate $10^8$ statically exact, independent stochastic realizations of the FPT $\tau$ using the simulation algorithms described in Ref.~\cite{lawley2021pdmp}. From these realizations, we then construct the empirical survival probability $S(t)=\P(\tau>t)$. We then construct the survival probability of the {fFPT} $T_1$ via
\begin{align*}
    \P(T_1>t)
    =\exp\bigg[-\lambda\int_0^t (1-S(s))\,\dd s\bigg].
\end{align*}
The mean of $T_1$ is then computed as the integral,
\begin{align*}
    \E[T_1]
    =\int_0^\infty \P(T_1>t)\,\dd t.
\end{align*}

\section*{Data availability statement}

The data that supports the findings of this study are available from the corresponding author upon reasonable request.

\section*{Author declarations}

The authors have no conflicts to disclose.

\end{document}